\documentclass[12pt,twoside,reqno]{amsart}

\usepackage{version}         % include, exclude and comment
 \usepackage{color}

\usepackage{multirow}
\usepackage[OT1]{fontenc}
\usepackage{type1cm}
\usepackage{amssymb}
\usepackage{mathrsfs}
\usepackage[english]{babel}
\usepackage[latin1]{inputenc}
\usepackage[T1]{fontenc}
\usepackage{palatino}
\usepackage{amsfonts}
\usepackage{amsmath}
\usepackage{amssymb}
\usepackage{amsthm}
\usepackage{bbm}
\usepackage{extarrows}
\usepackage{graphicx}
\usepackage[usenames,dvipsnames]{xcolor}
\usepackage{wrapfig}
\usepackage{type1cm}
\usepackage[bf]{caption}
\usepackage{esint}
\usepackage{version}
\usepackage[colorlinks = true, citecolor = black, linkcolor = black, urlcolor = black, pdfstartview=FitH]{hyperref}
\usepackage{caption}
\usepackage{tikz}
\usepackage{bbding, wasysym}
\usepackage{enumitem}
\usepackage{float}

\usepackage[left=2.3cm,top=3cm,right=2.3cm]{geometry}
\numberwithin{equation}{section}

\theoremstyle{plain}
\newtheorem{theorem}{Theorem}[section]

\newtheorem{lemma}[theorem]{Lemma}

\theoremstyle{remark}

\newtheorem*{ack}{Acknowledgement}

\theoremstyle{definition}
\newtheorem{definition}[theorem]{Definition}
\newtheorem{question}[theorem]{Question}

\newcommand{\R}{\mathbb{R}}

\newcommand{\N}{\mathbb{N}}

\newcommand{\cN}{\mathcal{N}}
\newcommand{\cL}{\mathcal{L}}
\newcommand{\cU}{\mathcal{U}}

\newcommand{\cF}{\mathcal{F}}

\newcommand{\cA}{\mathcal{A}}

\renewcommand{\epsilon}{\varepsilon}
\renewcommand{\rho}{\varrho}
\renewcommand{\phi}{\varphi}

\begin{document}

\title[Digit frequencies of beta-expansions]{Digit frequencies of beta-expansions}

\author{Yao-Qiang Li}
\address{Institut de Math\'ematiques de Jussieu-Paris Rive Gauche \\
         Sorbonne Universit\'e \\
         Paris, 75005 \\
         France}
\address{School of Mathematics \\
         South China University of Technology \\
         Guangzhou, 510641 \\
         P.R. China}
\email{yaoqiang.li@imj-prg.fr\quad yaoqiang.li@etu.upmc.fr}
\email{scutyaoqiangli@qq.com\quad scutyaoqiangli@gmail.com}

%\thanks{???????????????????}
\date{\today}

\begin{abstract}
Let $\beta>1$ be a non-integer. First we show that Lebesgue almost every number has a $\beta$-expansion of a given frequency if and only if Lebesgue almost every number has infinitely many $\beta$-expansions of the same given frequency. Then we deduce that Lebesgue almost every number has infinitely many balanced $\beta$-expansions, where an infinite sequence on the finite alphabet $\{0,1,\cdots,m\}$ is called balanced if the frequency of the digit $k$ is equal to the frequency of the digit $m-k$ for all $k\in\{0,1,\cdots,m\}$. Finally we consider variable frequency and prove that for every pseudo-golden ratio $\beta\in(1,2)$, there exists a constant $c=c(\beta)>0$ such that for any $p\in[\frac{1}{2}-c,\frac{1}{2}+c]$, Lebesgue almost every $x$ has infinitely many $\beta$-expansions with frequency of zeros equal to $p$.
\end{abstract}

\maketitle

\section{introduction}

To represent real numbers, the most common way is to use expansions in integer bases, especially in base $2$ or $10$. As a natural generalization, expansions in non-integer bases were introduced by R\'enyi \cite{R57} in 1957, and then attracted a lot of attention until now (see for examples \cite{ACS09,AC01,B89,BL16,FW12,FS92,LW08,LL18,P60,S97,S80}). They are known as beta-expansions nowadays.

Let $\N=\{1,2,3,\cdots\}$ be the set of positive integers and $\R$ be the set of real numbers. For $\beta>1$, we define the alphabet by
$$\cA_\beta=\left\{\begin{array}{ll}
\{0,1,\cdots,\beta-1\} & \text{if } \beta\in\N, \\
\{0,1,\cdots,\lfloor\beta\rfloor\} & \text{if } \beta\notin\N,
\end{array}\right.$$
where $\lfloor\beta\rfloor$ denotes the greatest integer no larger than $\beta$, and similarly we use $\lceil\beta\rceil$ to denote the smallest integer no less than $\beta$ throughout this paper. Let $x\in\R$. A sequence $(\epsilon_i)_{i\ge1}\in\cA_\beta^\N$ is called a \textit{$\beta$-expansion} of $x$ if
$$x=\sum_{i=1}^\infty\frac{\epsilon_i}{\beta^i}.$$
For $\beta>1$, let $I_\beta$ be the interval $[0,1]$ if $\beta\in\N$ and $[0,\frac{\lfloor\beta\rfloor}{\beta-1}]$ if $\beta\notin\N$, and let $I_\beta^o$ be the interior of $I_\beta$ (i.e. $I_\beta^o=(0,1)$ if $\beta\in\N$ and $I_\beta^o=(0,\frac{\lfloor\beta\rfloor}{\beta-1})$ if $\beta\notin\N$ ). It is straightforward to check that $x$ has a $\beta$-expansion if and only if $x\in I_\beta$. An interesting phenomenon is that an $x$ may have many $\beta$-expansions. For examples, \cite[Theorem 3]{EJK90} shows that if $\beta\in(1,\frac{1+\sqrt{5}}{2})$, every $x\in I_\beta^o$ has a continuum of different $\beta$-expansions, and \cite[Theorem 1]{S03} shows that if $\beta\in(1,2)$, Lebesgue almost every $x\in I_\beta$ has a continuum of different $\beta$-expansions. For more on the cardinality of $\beta$-expansions, we refer the reader to \cite{BS14,EJK94,GS01}.

In this paper we focus on the digit frequencies of $\beta$-expansions, which is a classical research topic. For examples, Borel's normal number theorem \cite{B09} says that for any integer $\beta>1$, Lebesgue almost every $x\in[0,1]$ has a $\beta$-expansion in which every finite word on $\cA_\beta$ with length $k$ occurs with frequency $\beta^{-k}$; Eggleston \cite{E49} proved that for each $p\in[0,1]$, the Hausdorff dimension (see \cite{F14} for definition) of the set, consisting of those $x\in[0,1]$ having a binary expansion with frequency of zeros equal to $p$, is equal to $(-p\log p-(1-p)\log(1-p))/(\log2)$. Let $\beta_T\approx1.80194$ be the unique zero in $(1,2]$ of the polynomial $x^3-x^2-2x+1$. Recently, on the one hand, Baker and Kong \cite{BK18} proved that if $\beta\in(1,\beta_T]$, then every $x\in I_\beta^o$ has a \textit{simply normal} $\beta$-expansion (i.e., the frequency of each digit is the same), and on the other hand, Jordan, Shmerkin and Solomyak \cite{JSS11} prove that if $\beta\in(\beta_T,2]$, then there exists $x\in I_\beta^o$ which does not have any simply normal $\beta$-expansions.

Let $m\in\N$. For any sequence $(\epsilon_i)_{i\ge1}\in\{0,1,\cdots,m\}^\N$, we define the \textit{upper-frequency}, \textit{lower-frequency} and \textit{frequency} of the digit $k$ by
$$\overline{\text{Freq}}_k(\epsilon_i):=\varlimsup_{n\to\infty}\frac{\sharp\{1\le i\le n:\epsilon_i=k\}}{n},$$
$$\underline{\text{Freq}}_k(\epsilon_i):=\varliminf_{n\to\infty}\frac{\sharp\{1\le i\le n:\epsilon_i=k\}}{n}$$
and
$$\text{Freq}_k(\epsilon_i):=\lim_{n\to\infty}\frac{\sharp\{1\le i\le n:\epsilon_i=k\}}{n}$$
(assuming the limit exists) respectively, where $\sharp$ denotes the cardinality. If $\overline{p}=(\overline{p}_0,\cdots,\overline{p}_m)$, $\underline{p}=(\underline{p}_0,\cdots,\underline{p}_m)\in[0,1]^m$ satisfy
$$\overline{\text{Freq}}_k(\epsilon_i)=\overline{p}_k\quad\text{and}\quad\underline{\text{Freq}}_k(\epsilon_i)=\underline{p}_k\quad\text{for all }k\in\{0,1,\cdots,m\},$$
we say that $(\epsilon_i)_{i\ge1}$ is of frequency $(\overline{p},\underline{p})$.

The following theorem is the first main result in this paper.

\begin{theorem}\label{main1}
For all $\beta\in(1,+\infty)\setminus\N$ and $\overline{p},\underline{p}\in[0,1]^{\lceil\beta\rceil}$, Lebesgue almost every $x\in I_\beta$ has a $\beta$-expansion of frequency $(\overline{p},\underline{p})$ if and only if Lebesgue almost every $x\in I_\beta$ has infinitely many $\beta$-expansions of frequency $(\overline{p},\underline{p})$.
\end{theorem}

As the second main result, the next theorem focuses on a special kind of frequency. Let $m\in\N$. A sequence $(\epsilon_i)_{i\ge1}\in\{0,1,\cdots,m\}^\N$ is called \textit{balanced} if Freq$_k(\epsilon_i)=$Freq$_{m-k}(\epsilon_i)$ for all $k\in\{0,1,\cdots,m\}$.

\begin{theorem}\label{main2}
For all $\beta\in(1,+\infty)\setminus\N$, Lebesgue almost every $x\in I_\beta$ has infinitely many balanced $\beta$-expansions.
\end{theorem}

In the following, we consider variable frequency. Recently, Baker proved in \cite{B18} that for any $\beta\in(1,\frac{1+\sqrt5}{2})$, there exists $c=c(\beta)>0$ such that for any $p\in[\frac{1}{2}-c,\frac{1}{2}+c]$ and $x\in I_\beta^o$, there exists a $\beta$-expansion of $x$ with frequency of zeros equal to $p$. This result is sharp, since for any $\beta\in[\frac{1+\sqrt5}{2},2)$, there exists an $x\in I_\beta^o$ such that for any $\beta$-expansion of $x$ its frequency of zeros exists and is equal to either $0$ or $\frac{1}{2}$ (see the statements between Theorem 1.1 and Theorem 1.2 in \cite{BK18}). It is natural to ask for which $\beta\in[\frac{1+\sqrt5}{2},2)$, the result can be true for almost every $x\in I_\beta^o$. We give a class of such $\beta$ in Theorem \ref{main3} as the third main result in this paper. They are the \textit{pseudo-golden ratios}, i.e., the $\beta\in(1,2)$ such that $\beta^m-\beta^{m-1}-\cdots-\beta-1=0$ for some integer $m\ge2$. Note that the smallest pseudo-golden ratio is the golden ratio $\frac{1+\sqrt5}{2}$.

\begin{theorem}\label{main3}
Let $\beta\in(1,2)$ such that $\beta^m-\beta^{m-1}-\cdots-\beta-1=0$ for some integer $m\ge2$ and let $c=\frac{(m-1)(2-\beta)}{2(m\beta+\beta-2m)}$ $(>0)$. Then for any $p\in[\frac{1}{2}-c,\frac{1}{2}+c]$,
Lebesgue almost every $x\in I_\beta$ has infinitely many $\beta$-expansions with frequency of zeros equal to $p$.
\end{theorem}

We give some notations and preliminaries in the next section, prove the main results in Section 3 and end this paper with further questions in the last section.

\section{notations and preliminaries}

Let $\beta>1$. We define the maps $T_k(x):=\beta x-k$ for $x\in\R$ and $k\in\N\cup\{0\}$. Given $x\in I_\beta$, let
$$\Sigma_\beta(x):=\Big\{(\epsilon_i)_{i\ge1}\in\cA_\beta^\N:\sum_{i=1}^\infty\frac{\epsilon_i}{\beta^i}=x\Big\}$$
and
$$\Omega_\beta(x):=\Big\{(a_i)_{i\ge1}\in\{T_k,k\in\cA_\beta\}^\N:(a_n\circ\cdots\circ a_1)(x)\in I_\beta\text{ for all }n\in\N\Big\}.$$
The following lemma given by Baker is a dynamical interpretation of $\beta$-expansions.

\begin{lemma}[\cite{B14,B18}]\label{bijection}
For any $x\in I_\beta$, we have $\sharp\Sigma_\beta(x)=\sharp\Omega_\beta(x)$. Moreover, the map which sends $(\epsilon_i)_{i\ge1}$ to $(T_{\epsilon_i})_{i\ge1}$ is a bijection between $\Sigma_\beta(x)$ and $\Omega_\beta(x)$.
\end{lemma}

We need the following concepts and the well known Birkhoff's Ergodic Theorem in the proof of our main results.

\begin{definition}[Absolute continuity and equivalence]
Let $\mu$ and $\nu$ be measures on a measurable space $(X,\cF)$. We say that $\mu$ is \textit{absolutely continuous} with respect to $\nu$ and denote it by $\mu \ll \nu$ if, for any $A \in\cF$, $\nu(A) = 0$ implies $\mu(A) = 0$. Moreover, if $\mu\ll\nu$ and $\nu\ll\mu$ we say that $\mu$ and $\nu$ are \textit{equivalent} and denote this property by $\mu\sim\nu$.
\end{definition}

\begin{theorem}[\cite{W82} Birkhoff's Ergodic Theorem]
Let $(X,\cF,\mu,T)$ be a measure-preserving dynamical system where the probability measure $\mu$ is ergodic with respect to $T$. Then for any real-valued integrable function $f:X\to\R$, we have
$$\lim_{n\to\infty}\frac{1}{n}\sum_{k=0}^{n-1}f(T^kx)=\int f d\mu$$
for $\mu$-a.e. (almost every) $x\in X$.
\end{theorem}

\section{Proof of the main results}

\begin{proof}[Proof of Theorem \ref{main1}] The ``if'' part is obvious. We only need to prove the ``only if'' part. Let $\cL$ be the Lebesgue measure. Suppose that $\cL$-a.e. $x\in I_\beta$ has a $\beta$-expansion of frequency $(\overline{p},\underline{p})$. Let
$$\cU_\beta:=\Big\{x\in I_\beta: x\text{ has a unique }\beta\text{-expansion}\Big\}$$
and
$$\cN_\beta^{\overline{p},\underline{p}}:=\Big\{x\in I_\beta:x\text{ has no }\beta\text{-expansions of frequency }(\overline{p},\underline{p})\Big\}.$$
On the one hand, it is well known that $\cL(\cU_\beta)=0$ (see for examples \cite{DK09,KL14}). On the other hand, by condition we know $\cL(\cN_\beta^{\overline{p},\underline{p}})=0$. Let
$$\Psi:=\Big(\cU_\beta\cup\cN_\beta^{\overline{p},\underline{p}}\Big)\cup\bigcup_{n=1}^\infty\bigcup_{\epsilon_1,\cdots,\epsilon_n\in\cA_\beta}T_{\epsilon_n}^{-1}\circ\cdots\circ T_{\epsilon_1}^{-1}\Big(\cU_\beta\cup\cN_\beta^{\overline{p},\underline{p}}\Big).$$
Then $\cL(\Psi)=0$. Let $x\in I_\beta\setminus\Psi$. It suffices to prove that $x$ has infinitely many different $\beta$-expansions of frequency $(\overline{p},\underline{p})$.

Let $(\epsilon_i)_{i\ge1}$ be a $\beta$-expansions of $x$. Since $x\notin\Psi$ implies $x\notin\cU_\beta$, $x$ has another $\beta$-expansion $(w^{(1)}_i)_{i\ge1}$. There exists $n_1\in\N$ such that $w^{(1)}_1\cdots w^{(1)}_{n_1-1}=\epsilon_1\cdots\epsilon_{n_1-1}$ and $w^{(1)}_{n_1}\neq\epsilon_{n_1}$. By
$$T_{w^{(1)}_{n_1}}\circ T_{\epsilon_{n_1}-1}\circ\cdots\circ T_{\epsilon_1}x=T_{w^{(1)}_{n_1}}\circ\cdots\circ T_{w^{(1)}_1}x=\sum_{i=1}^\infty\frac{w^{(1)}_{n_1+i}}{\beta^i},$$
we know that $(w^{(1)}_{n_1+i})_{i\ge1}$ is a $\beta$-expansion of $T_{w^{(1)}_{n_1}}\circ T_{\epsilon_{n_1}-1}\circ\cdots\circ T_{\epsilon_1}x$. Since $x\notin\Psi$ implies $T_{w^{(1)}_{n_1}}\circ T_{\epsilon_{n_1}-1}\circ\cdots\circ T_{\epsilon_1}x\notin\cN_\beta^{\overline{p},\underline{p}}$, $T_{w^{(1)}_{n_1}}\circ T_{\epsilon_{n_1}-1}\circ\cdots\circ T_{\epsilon_1}x$ has a $\beta$-expansion $(\epsilon^{(1)}_{n_1+i})_{i\ge1}$ of frequency $(\overline{p},\underline{p})$. Let $\epsilon^{(1)}_1\cdots\epsilon^{(1)}_{n_1-1}\epsilon^{(1)}_{n_1}:=\epsilon_1\cdots\epsilon_{n_1-1}w^{(1)}_{n_1}$. Then $(\epsilon^{(1)}_i)_{i\ge1}$ is a $\beta$-expansion of $x$ of frequency $(\overline{p},\underline{p})$ with $\epsilon^{(1)}_{n_1}\neq\epsilon_{n_1}$, which implies that $(\epsilon_i)_{i\ge1}$ and $(\epsilon^{(1)}_i)_{i\ge1}$ are different.

Note that $(\epsilon_{n_1+i})_{i\ge1}$ is a $\beta$-expansion of $T_{\epsilon_{n_1}}\circ\cdots\circ T_{\epsilon_1}x$. Since $x\notin\Psi$ implies $T_{\epsilon_{n_1}}\circ\cdots\circ T_{\epsilon_1}x\notin\cU_\beta$, $T_{\epsilon_{n_1}}\circ\cdots\circ T_{\epsilon_1}x$ has another $\beta$-expansion $(w^{(2)}_{n_1+i})_{i\ge1}$. There exists $n_2>n_1$ such that $w^{(2)}_{n_1+1}\cdots w^{(2)}_{n_2-1}=\epsilon_{n_1+1}\cdots\epsilon_{n_2-1}$ and $w^{(2)}_{n_2}\neq\epsilon_{n_2}$. By
$$T_{w^{(2)}_{n_2}}\circ T_{\epsilon_{n_2}-1}\circ\cdots\circ T_{\epsilon_1}x=T_{w^{(2)}_{n_2}}\circ\cdots\circ T_{w^{(2)}_{n_1+1}}\circ(T_{\epsilon_{n_1}}\circ\cdots\circ T_{\epsilon_1}x)=\sum_{i=1}^\infty\frac{w^{(2)}_{n_2+i}}{\beta^i},$$
we know that $(w^{(2)}_{n_2+i})_{i\ge1}$ is a $\beta$-expansion of $T_{w^{(2)}_{n_2}}\circ T_{\epsilon_{n_2}-1}\circ\cdots\circ T_{\epsilon_1}x$. Since $x\notin\Psi$ implies $T_{w^{(2)}_{n_2}}\circ T_{\epsilon_{n_2}-1}\circ\cdots\circ T_{\epsilon_1}x\notin\cN_\beta^{\overline{p},\underline{p}}$, $T_{w^{(2)}_{n_2}}\circ T_{\epsilon_{n_2}-1}\circ\cdots\circ T_{\epsilon_1}x$ has a $\beta$-expansion $(\epsilon^{(2)}_{n_2+i})_{i\ge1}$ of frequency $(\overline{p},\underline{p})$. Let $\epsilon^{(2)}_1\cdots\epsilon^{(2)}_{n_2-1}\epsilon^{(2)}_{n_2}:=\epsilon_1\cdots\epsilon_{n_2-1}w^{(2)}_{n_2}$. Then $(\epsilon^{(2)}_i)_{i\ge1}$ is a $\beta$-expansion of $x$ of frequency $(\overline{p},\underline{p})$ with $\epsilon^{(2)}_{n_1}=\epsilon_{n_1}$ and $\epsilon^{(2)}_{n_2}\neq\epsilon_{n_2}$, which implies that $(\epsilon_i)_{i\ge1}$, $(\epsilon^{(1)}_i)_{i\ge1}$ and $(\epsilon^{(2)}_i)_{i\ge1}$ are all different.

$\cdots$

Generally, suppose that for some $j\in\N$ we have already constructed $(\epsilon^{(1)}_i)_{i\ge1}$, $(\epsilon^{(2)}_i)_{i\ge1}$, $\cdots$, $(\epsilon^{(j)}_i)_{i\ge1}$, which are all $\beta$-expansions of $x$ of frequency $(\overline{p},\underline{p})$ such that
$$\left\{\begin{array}{l}
\epsilon^{(1)}_{n_1}\neq\epsilon_{n_1},\\
\epsilon^{(2)}_{n_1}=\epsilon_{n_1},\epsilon^{(2)}_{n_2}\neq\epsilon_{n_2},\\
\epsilon^{(3)}_{n_1}=\epsilon_{n_1},\epsilon^{(3)}_{n_2}=\epsilon_{n_2},\epsilon^{(3)}_{n_3}\neq\epsilon_{n_3},\\
\cdots\\
\epsilon^{(j)}_{n_1}=\epsilon_{n_1},\epsilon^{(j)}_{n_2}=\epsilon_{n_2},\cdots,\epsilon^{(j)}_{n_{j-1}}\neq\epsilon_{n_{j-1}},\epsilon^{(j)}_{n_{j}}\neq\epsilon_{n_{j}}.
\end{array}\right.$$
Note that $(\epsilon_{n_j+i})_{i\ge1}$ is a $\beta$-expansion of $T_{\epsilon_{n_j}}\circ\cdots\circ T_{\epsilon_1}x$. Since $x\notin\Psi$ implies $T_{\epsilon_{n_j}}\circ\cdots\circ T_{\epsilon_1}x\notin\cU_\beta$, $T_{\epsilon_{n_j}}\circ\cdots\circ T_{\epsilon_1}x$ has another $\beta$-expansion $(w^{(j+1)}_{n_j+i})_{i\ge1}$. There exists $n_{j+1}>n_j$ such that $w^{(j+1)}_{n_j+1}\cdots w^{(j+1)}_{n_{j+1}-1}=\epsilon_{n_j+1}\cdots\epsilon_{n_{j+1}-1}$ and $w^{(j+1)}_{n_{j+1}}\neq\epsilon_{n_{j+1}}$. By
$$T_{w^{(j+1)}_{n_{j+1}}}\circ T_{\epsilon_{n_{j+1}}-1}\circ\cdots\circ T_{\epsilon_1}x=T_{w^{(j+1)}_{n_{j+1}}}\circ\cdots\circ T_{w^{(j+1)}_{n_j+1}}\circ(T_{\epsilon_{n_j}}\circ\cdots\circ T_{\epsilon_1}x)=\sum_{i=1}^\infty\frac{w^{(j+1)}_{n_{j+1}+i}}{\beta^i},$$
we know that $(w^{(j+1)}_{n_{j+1}+i})_{i\ge1}$ is a $\beta$-expansion of $T_{w^{(j+1)}_{n_{j+1}}}\circ T_{\epsilon_{n_{j+1}}-1}\circ\cdots\circ T_{\epsilon_1}x$. Since $x\notin\Psi$ implies $T_{w^{(j+1)}_{n_{j+1}}}\circ T_{\epsilon_{n_{j+1}}-1}\circ\cdots\circ T_{\epsilon_1}x\notin\cN_\beta^{\overline{p},\underline{p}}$, $T_{w^{(j+1)}_{n_{j+1}}}\circ T_{\epsilon_{n_{j+1}}-1}\circ\cdots\circ T_{\epsilon_1}x$ has a $\beta$-expansion $(\epsilon^{(j+1)}_{n_{j+1}+i})_{i\ge1}$ of frequency $(\overline{p},\underline{p})$. Let $\epsilon^{(j+1)}_1\cdots\epsilon^{(j+1)}_{n_{j+1}-1}\epsilon^{(j+1)}_{n_{j+1}}:=\epsilon_1\cdots\epsilon_{n_{j+1}-1}w^{(j+1)}_{n_{j+1}}$. Then $(\epsilon^{(j+1)}_i)_{i\ge1}$ is a $\beta$-expansion of $x$ of frequency $(\overline{p},\underline{p})$ with $\epsilon^{(j+1)}_{n_1}=\epsilon_{n_1},\cdots,\epsilon^{(j+1)}_{n_j}=\epsilon_{n_j}$ and $\epsilon^{(j+1)}_{n_{j+1}}\neq\epsilon_{n_{j+1}}$, which implies that $(\epsilon_i)_{i\ge1}$, $(\epsilon^{(1)}_i)_{i\ge1}$, $\cdots$, $(\epsilon^{(j+1)}_i)_{i\ge1}$ are all different.

$\cdots$

It follows from repeating the above process that $x$ has infinitely many different $\beta$-expansions of frequency $(\overline{p},\underline{p})$.
\end{proof}

Theorem \ref{main2} follows immediately from Theorem \ref{main1} and the following lemma.

\begin{lemma}\label{lemma}
For all $\beta>1$, Lebesgue almost every $x\in I_\beta$ has a balanced $\beta$-expansion.
\end{lemma}
\begin{proof} The conclusion follows from the well known Borel's Normal Number Theorem \cite{B09} if $\beta\in\N$ and follows from \cite[Theorem 4.1]{BK18} if $\beta\in(1,2)$. Thus we only need to consider $\beta>2$ with $\beta\notin\N$ in the following. Let
$$z_1:=\frac{1}{2}\big(\frac{\lfloor\beta\rfloor}{\beta-1}-\frac{\lfloor\beta\rfloor-1}{\beta}\big)\quad\text{and}\quad z_{k+1}:=z_k+\frac{1}{\beta}\quad\text{for all }k\in\{1,2,\cdots,\lfloor\beta\rfloor-1\}.$$
Define $T: I_\beta\to I_\beta$ by
$$T(x):=\left\{\begin{array}{ll}
T_0(x)=\beta x & \text{for } x\in[0,z_1),\\
T_k(x)=\beta x-k & \text{for } x\in[z_k,z_{k+1})\text{ and }k\in\{1,2,\cdots,\lfloor\beta\rfloor-1\},\\
T_{\lfloor\beta\rfloor}(x)=\beta x-\lfloor\beta\rfloor & \text{for } x\in[z_{\lfloor\beta\rfloor},\frac{\lfloor\beta\rfloor}{\beta-1}].
\end{array}\right.$$
Let
$$z_0:=\frac{\lfloor\beta\rfloor}{2(\beta-1)}-\frac{1}{2}\quad\text{and}\quad z_{\lceil\beta\rceil}:=z_0+1=\frac{\lfloor\beta\rfloor}{2(\beta-1)}+\frac{1}{2}.$$
Then $T_1(z_1)=T_2(z_2)=\cdots=T_{\lfloor\beta\rfloor}(z_{\lfloor\beta\rfloor})=z_0$ and $T_0(z_1)=T_1(z_2)=\cdots=T_{\lfloor\beta\rfloor-1}(z_{\lfloor\beta\rfloor})=z_{\lceil\beta\rceil}$.

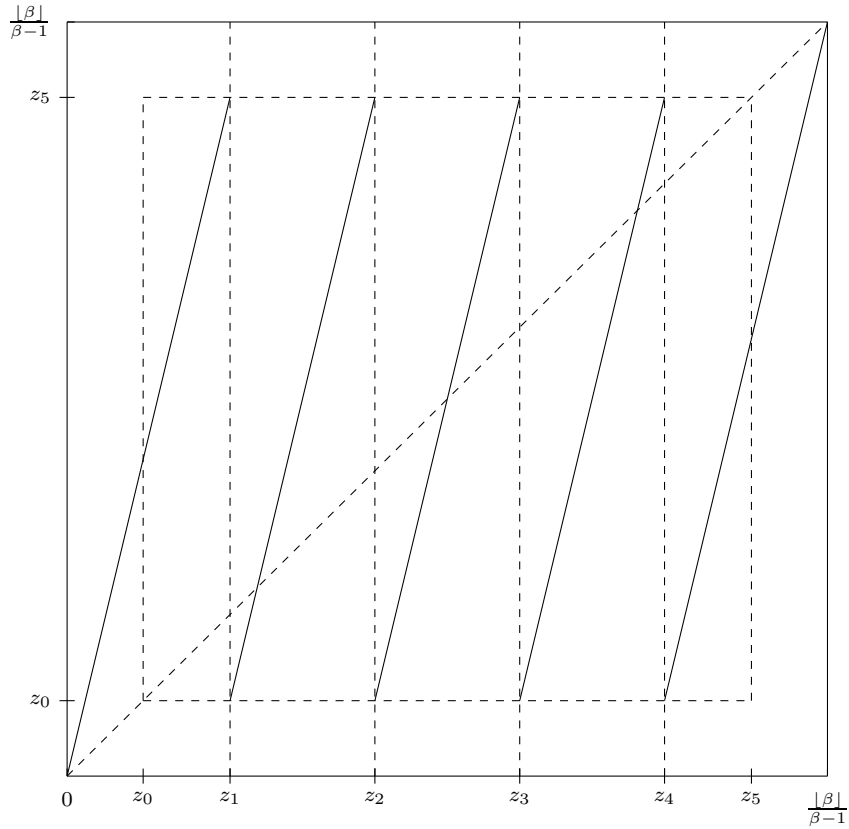
\begin{figure}[H]\label{T}
\begin{tikzpicture}
\draw[-](0,0)--(10,0);
\draw[-](0,0)--(0,10);
\draw[-](0,10)--(10,10);
\draw[-](10,0)--(10,10);
\draw[-](0,0)--(15/7,9);
\draw[-](15/7,1)--(85/21,9);
\draw[-](85/21,1)--(125/21,9);
\draw[-](125/21,1)--(55/7,9);
\draw[-](55/7,1)--(10,10);
\draw[dashed](15/7,0)--(15/7,10);
\draw[dashed](85/21,0)--(85/21,10);
\draw[dashed](125/21,0)--(125/21,10);
\draw[dashed](55/7,0)--(55/7,10);
\draw[dashed](0,0)--(10,10);
\draw[dashed](1,1)--(9,1);
\draw[dashed](1,1)--(1,9);
\draw[dashed](1,9)--(9,9);
\draw[dashed](9,1)--(9,9);
\foreach \x in {0}{\draw(\x,0)--(\x,0.1)node[below,outer sep=2pt,font=\tiny]at(\x,0){$0$};}
\foreach \x in {1}{\draw(\x,-0.1)--(\x,0.1)node[below,outer sep=2pt,font=\tiny]at(\x,0){$z_0$};}
\foreach \x in {15/7}{\draw(\x,-0.1)--(\x,0.1)node[below,outer sep=2pt,font=\tiny]at(\x,0){$z_1$};}
\foreach \x in {85/21}{\draw(\x,-0.1)--(\x,0.1)node[below,outer sep=2pt,font=\tiny]at(\x,0){$z_2$};}
\foreach \x in {125/21}{\draw(\x,-0.1)--(\x,0.1)node[below,outer sep=2pt,font=\tiny]at(\x,0){$z_3$};}
\foreach \x in {55/7}{\draw(\x,-0.1)--(\x,0.1)node[below,outer sep=2pt,font=\tiny]at(\x,0){$z_4$};}
\foreach \x in {9}{\draw(\x,-0.1)--(\x,0.1)node[below,outer sep=2pt,font=\tiny]at(\x,0){$z_5$};}
\foreach \x in {10}{\draw(\x,0)--(\x,0.1)node[below,outer sep=2pt,font=\tiny]at(\x,0){$\frac{\lfloor\beta\rfloor}{\beta-1}$};}
\foreach \y in {1}{\draw(-0.1,\y)--(0.1,\y)node[left,outer sep=2pt,font=\tiny]at(0,\y){$z_0$};}
\foreach \y in {9}{\draw(-0.1,\y)--(0.1,\y)node[left,outer sep=2pt,font=\tiny]at(0,\y){$z_5$};}
\foreach \y in {10}{\draw(0,\y)--(0.1,\y)node[left,outer sep=2pt,font=\tiny]at(0,\y){$\frac{\lfloor\beta\rfloor}{\beta-1}$};}
\end{tikzpicture}
\caption{The graph of $T$ for some $\beta\in(4,5)$.}
\end{figure}

We consider the restriction $T|_{[z_0,z_{\lceil\beta\rceil})}:[z_0,z_{\lceil\beta\rceil})\to[z_0,z_{\lceil\beta\rceil})$. By Theorem 5.2 in \cite{W75}, there exists a $T|_{[z_0,z_{\lceil\beta\rceil})}$-invariant ergodic Borel probability measure $\mu$ on $[z_0,z_{\lceil\beta\rceil})$ equivalent to the Lebesgue measure $\cL$. For any $x\in[z_0,z_{\lceil\beta\rceil})$ which is not a preimage of a discontinuity point of $T|_{[z_0,z_{\lceil\beta\rceil})}$, by symmetry, we know that for any $k\in\{0,1,\cdots,\lfloor\beta\rfloor\}$ and $i\in\{0,1,2,\cdots\}$,
$$T^i(x)\in(z_k,z_{k+1})\Leftrightarrow T^i\Big(\frac{\lfloor\beta\rfloor}{\beta-1}-x\Big)\in(z_{\lfloor\beta\rfloor-k},z_{\lceil\beta\rceil-k}).$$
For all $k\in\{0,1,\cdots,\lfloor\beta\rfloor\}$, it follows from Birkhoff's Ergodic Theorem that for $\cL$-a.e. $x\in[z_0,z_{\lceil\beta\rceil})$,
\begin{eqnarray}
\label{Birkhoff}\mu((z_k,z_{k+1}))=\int_{z_0}^{z_{\lceil\beta\rceil}}\mathbbm{1}_{(z_k,z_{k+1})}d\mu=\lim_{n\to\infty}\frac{1}{n}\sum_{i=0}^{n-1}\mathbbm{1}_{(z_k,z_{k+1})}\Big(T^i(x)\Big)\\
\label{symmetry}=\lim_{n\to\infty}\frac{1}{n}\sum_{i=0}^{n-1}\mathbbm{1}_{(z_{\lfloor\beta\rfloor-k},z_{\lceil\beta\rceil-k})}\Big(T^i\Big(\frac{\lfloor\beta\rfloor}{\beta-1}-x\Big)\Big)
\end{eqnarray}
and for $\cL$-a.e. $y\in[z_0,z_{\lceil\beta\rceil})$,
$$\mu((z_{\lfloor\beta\rfloor-k},z_{\lceil\beta\rceil-k}))=\int_{z_0}^{z_{\lceil\beta\rceil}}\mathbbm{1}_{(z_{\lfloor\beta\rfloor-k},z_{\lceil\beta\rceil-k})}d\mu=\lim_{n\to\infty}\frac{1}{n}\sum_{i=0}^{n-1}\mathbbm{1}_{(z_{\lfloor\beta\rfloor-k},z_{\lceil\beta\rceil-k})}\Big(T^i(y)\Big),$$
which implies that for $\cL$-a.e. $(\frac{\lfloor\beta\rfloor}{\beta-1}-x)\in(z_0,z_{\lceil\beta\rceil})$,
$$\mu((z_{\lfloor\beta\rfloor-k},z_{\lceil\beta\rceil-k}))=\lim_{n\to\infty}\frac{1}{n}\sum_{i=0}^{n-1}\mathbbm{1}_{(z_{\lfloor\beta\rfloor-k},z_{\lceil\beta\rceil-k})}\Big(T^i\Big(\frac{\lceil\beta\rceil}{\beta-1}-x\Big)\Big).$$
So this is also true for $\cL$-a.e $x\in(z_0,z_{\lceil\beta\rceil})$. Recall (\ref{symmetry}), we get
\begin{eqnarray}\label{=}
\mu((z_k,z_{k+1}))=\mu((z_{\lfloor\beta\rfloor-k},z_{\lceil\beta\rceil-k}))\quad\text{for }k\in\{0,1,\cdots,\lfloor\beta\rfloor\}.
\end{eqnarray}
For every $x\in I_\beta$, define a sequence $(\epsilon_i(x))_{i\ge1}\in\{0,1,\cdots,\lfloor\beta\rfloor\}^\N$ by
$$\epsilon_i(x):=\left\{\begin{array}{ll}
0 & \text{if } T^{i-1}x\in[0,z_1),\\
k & \text{if } T^{i-1}x\in[z_k,z_{k+1})\text{ for some }k\in\{1,2,\cdots,\lfloor\beta\rfloor-1\},\\
\lfloor\beta\rfloor & \text{if } T^{i-1}x\in[z_{\lfloor\beta\rfloor},\frac{\lfloor\beta\rfloor}{\beta-1}].
\end{array}\right.$$
Then for all $k\in\{0,1,\cdots,\lfloor\beta\rfloor\}$, $i\in\{0,1,2,\cdots\}$ and $x\in[z_0,z_{\lceil\beta\rceil})$,
$$\mathbbm{1}_{[z_k,z_{k+1})}(T^ix)=1\Leftrightarrow T^ix\in[z_k,z_{k+1})\Leftrightarrow\epsilon_{i+1}(x)=k.$$
By (\ref{Birkhoff}), we know that for all $k\in\{0,1,\cdots,\lfloor\beta\rfloor\}$ and $\cL$-a.e. $x\in[z_0,z_{\lceil\beta\rceil})$,
$$\text{Freq}_k(\epsilon_i(x))=\lim_{n\to\infty}\frac{\sharp\{1\le i\le n:\epsilon_i(x)=k\}}{n}=\mu((z_k,z_{k+1})).$$
It follows from (\ref{=}) that that for all $k\in\{0,1,\cdots,\lfloor\beta\rfloor\}$ and $\cL$-a.e. $x\in[z_0,z_{\lceil\beta\rceil})$,
\begin{eqnarray}\label{balanced}
\text{Freq}_k(\epsilon_i(x))=\text{Freq}_{\lfloor\beta\rfloor-k}(\epsilon_i(x)).
\end{eqnarray}
\begin{itemize}
\item[(1)] For any $x\in I_\beta$, we prove that $(\epsilon_i(x))_{i\ge1}$ is a $\beta$-expansion of $x$, i.e., $\sum_{i=1}^\infty\frac{\epsilon_i(x)}{\beta^i}=x$. In fact, by Lemma \ref{bijection}, it suffices to show $T_{\epsilon_n(x)}\circ\cdots\circ T_{\epsilon_1(x)}(x)\in I_\beta$ for all $n\in\N$. We only need to prove $T_{\epsilon_n(x)}\circ\cdots\circ T_{\epsilon_1(x)}(x)=T^n(x)$ by induction as follows. Let $n=1$.
    \begin{itemize}
    \item[\textcircled{\footnotesize{$1$}}] If $x\in[0,z_1)$, then $\epsilon_1(x)=0$ and $T_{\epsilon_1(x)}(x)=T_0(x)=T(x)$.
    \item[\textcircled{\footnotesize{$2$}}] If $x\in[z_k,z_{k+1})$ for some $k\in\{1,2,\cdots,\lfloor\beta\rfloor-1\}$, then $\epsilon_1(x)=k$ and $T_{\epsilon_1(x)}(x)=T_k(x)=T(x)$.
    \item[\textcircled{\footnotesize{$3$}}] If $x\in[z_{\lfloor\beta\rfloor},\frac{\lfloor\beta\rfloor}{\beta-1}]$, then $\epsilon_1(x)=\lfloor\beta\rfloor$ and $T_{\epsilon_1(x)}(x)=T_{\lfloor\beta\rfloor}(x)=T(x)$.
    \end{itemize}
    Assumes that for some $n\in\N$ we have $T_{\epsilon_n(x)}\circ\cdots\circ T_{\epsilon_1(x)}(x)=T^n(x)$.
    \begin{itemize}
    \item[\textcircled{\footnotesize{$1$}}] If $T^n(x)\in[0,z_1)$, then $\epsilon_{n+1}(x)=0$ and
    $$T_{\epsilon_{n+1}(x)}\circ T_{\epsilon_n(x)}\circ\cdots\circ T_{\epsilon_1(x)}(x)=T_0\circ T^n(x)=T^{n+1}(x).$$
    \item[\textcircled{\footnotesize{$2$}}] If $T^n(x)\in[z_k,z_{k+1})$ for some $k\in\{1,2,\cdots,\lfloor\beta\rfloor-1\}$, then $\epsilon_{n+1}(x)=k$ and
    $$T_{\epsilon_{n+1}(x)}\circ T_{\epsilon_n(x)}\circ\cdots\circ T_{\epsilon_1(x)}(x)=T_k\circ T^n(x)=T^{n+1}(x).$$
    \item[\textcircled{\footnotesize{$3$}}] If $T^n(x)\in[z_{\lfloor\beta\rfloor},\frac{\lfloor\beta\rfloor}{\beta-1}]$, then $\epsilon_{n+1}(x)=\lfloor\beta\rfloor$ and
    $$T_{\epsilon_{n+1}(x)}\circ T_{\epsilon_n(x)}\circ\cdots\circ T_{\epsilon_1(x)}(x)=T_{\lfloor\beta\rfloor}\circ T^n(x)=T^{n+1}(x).$$
    \end{itemize}
\end{itemize}
Combining (1) and (\ref{balanced}), we know that $\cL$-a.e. $x\in[z_0,z_{\lceil\beta\rceil}]$ has a balanced $\beta$-expansion. Let
$$N:=\big\{x\in I_\beta: x \text{ has no balanced }\beta\text{-expansions}\big\}.$$
We have already proved $\cL(N\cap[z_0,z_{\lceil\beta\rceil}])=0$. To end the proof of this lemma, we need to show $\cL(N)=0$. In fact, it suffices to prove $\cL(N\cap(0,z_0))=\cL(N\cap(z_{\lceil\beta\rceil},\frac{\lfloor\beta\rfloor}{\beta-1}))=0$.
\begin{itemize}
\item[i)] Prove $\cL(N\cap(0,z_0))=0$.
\newline By $\cL(N\cap[z_0,z_{\lceil\beta\rceil}])=0$, we know that for any $n\in\N$, $\cL(T_0^{-n}(N\cap[z_0,z_{\lceil\beta\rceil}]))=0$. It suffices to prove $N\cap(0,z_0)\subset\bigcup_{n=1}^\infty T_0^{-n}(N\cap[z_0,z_{\lceil\beta\rceil}])$.
\newline (By contradiction) Let $x\in N\cap(0,z_0)$ and assume $x\notin\bigcup_{n=1}^\infty T_0^{-n}(N\cap[z_0,z_{\lceil\beta\rceil}])$. By $x\in(0,z_0)$, one can verify that there exists $k\ge1$ such that $T_0^kx\in[z_0,z_{\lceil\beta\rceil}]$. Since $x\notin T_0^{-k}(N\cap[z_0,z_{\lceil\beta\rceil}])$, we must have $T_0^kx\notin N$. This means that there exists a balanced sequence $(w_i)_{i\ge1}\in\cA_\beta^\N$ such that $T_0^kx=\sum_{i=1}^\infty\frac{w_i}{\beta^i}$, and then
$$x=\frac{0}{\beta}+\frac{0}{\beta^2}+\cdots+\frac{0}{\beta^k}+\sum_{i=1}^\infty\frac{w_i}{\beta^{k+i}}=:\sum_{i=1}^\infty\frac{\epsilon_i}{\beta^i}$$
where $\epsilon_1=\cdots=\epsilon_k:=0$ and $\epsilon_{k+i}:=w_i$ for $i\ge1$. It follows that $(\epsilon_i)_{i\ge1}$ is a balanced $\beta$-expansion of $x$, which contradicts $x\in N$.
\item[ii)] Prove $\cL(N\cap(z_{\lceil\beta\rceil},\frac{\lfloor\beta\rfloor}{\beta-1}))=0$.
\newline By $\cL(N\cap[z_0,z_{\lceil\beta\rceil}])=0$, we know that for any $n\in\N$, $\cL(T_{\lfloor\beta\rfloor}^{-n}(N\cap[z_0,z_{\lceil\beta\rceil}]))=0$. It suffices to prove $N\cap(z_{\lceil\beta\rceil},\frac{\lfloor\beta\rfloor}{\beta-1})\subset\bigcup_{n=1}^\infty T_{\lfloor\beta\rfloor}^{-n}(N\cap[z_0,z_{\lceil\beta\rceil}])$.
\newline (By contradiction) Let $x\in N\cap(z_{\lceil\beta\rceil},\frac{\lfloor\beta\rfloor}{\beta-1})$ and assume $x\notin\bigcup_{n=1}^\infty T_{\lfloor\beta\rfloor}^{-n}(N\cap[z_0,z_{\lceil\beta\rceil}])$. By $x\in(z_{\lceil\beta\rceil},\frac{\lfloor\beta\rfloor}{\beta-1})$, one can verify that there exists $k\ge1$ such that $T_{\lfloor\beta\rfloor}^kx\in[z_0,z_{\lceil\beta\rceil}]$. Since $x\notin T_{\lfloor\beta\rfloor}^{-k}(N\cap[z_0,z_{\lceil\beta\rceil}])$, we must have $T_{\lfloor\beta\rfloor}^kx\notin N$. This means that there exists a balanced sequence $(w_i)_{i\ge1}\in\cA_\beta^\N$ such that $T_{\lfloor\beta\rfloor}^kx=\sum_{i=1}^\infty\frac{w_i}{\beta^i}$, and then
$$x=\frac{\lfloor\beta\rfloor}{\beta}+\frac{\lfloor\beta\rfloor}{\beta^2}+\cdots+\frac{\lfloor\beta\rfloor}{\beta^k}+\sum_{i=1}^\infty\frac{w_i}{\beta^{k+i}}=:\sum_{i=1}^\infty\frac{\epsilon_i}{\beta^i}$$
where $\epsilon_1=\cdots=\epsilon_k:=\lfloor\beta\rfloor$ and $\epsilon_{k+i}:=w_i$ for $i\ge1$. It follows that $(\epsilon_i)_{i\ge1}$ is a balanced $\beta$-expansion of $x$, which contradicts $x\in N$.
\end{itemize}
\end{proof}

\begin{proof}[Proof of Theorem \ref{main3}]
Let $\beta\in(1,2)$ such that $\beta^m-\beta^{m-1}-\cdots-\beta-1=0$ for some integer $m\ge2$ and let $c=\frac{(m-1)(2-\beta)}{2(m\beta+\beta-2m)}$. We have $c>0$ since $m-1>0$, $2-\beta>0$ and $m\beta+\beta-2m>0$, which is a consequence of
$$m+1<2m<2(\beta^{m-1}+\cdots+\beta+1)=2\beta^m=\frac{2}{2-\beta},$$
where the equalities follows from
$$\beta^m=\beta^{m-1}+\cdots+\beta+1=\frac{\beta^m-1}{\beta-1}.$$

For any $x\in[0,\frac{1}{\beta-1}-1]$, define
$$f(x):=\frac{(\beta-1)(1-(m-1)x)}{m\beta+\beta-2m}.$$
Then
$$f(0)=\frac{\beta-1}{m\beta+\beta-2m}=\frac{1}{2}+c\quad\text{and}\quad f(\frac{1}{\beta-1}-1)=\frac{m\beta+1-2m}{m\beta+\beta-2m}=\frac{1}{2}-c,$$
i.e., $[f(\frac{1}{\beta-1}-1),f(0)]=[\frac{1}{2}-c,\frac{1}{2}+c]$. Since $f$ is continuous, for any $p\in[\frac{1}{2}-c,\frac{1}{2}+c]$, there exists $b\in[0,\frac{1}{\beta-1}-1]$ such that $f(b)=p$. We only consider $b\in[0,\frac{1}{\beta-1}-1)$ in the following, since the proof for the case $b\in(0,\frac{1}{\beta-1}-1]$ is similar. Define $T: I_\beta\to I_\beta$ by
$$T(x):=\left\{\begin{array}{ll}
T_0(x)=\beta x & \text{for } x\in[0,\frac{b+1}{\beta}),\\
T_1(x)=\beta x-1 & \text{for } x\in[\frac{b+1}{\beta},\frac{1}{\beta-1}].
\end{array}\right.$$
\begin{figure}[H]\label{T}
\begin{tikzpicture}
\draw[-](0,0)--(10,0);
\draw[-](0,0)--(0,10);
\draw[-](0,10)--(10,10);
\draw[-](10,0)--(10,10);
\draw[-](0,0)--(77/17,7.7);
\draw[-](77/17,0.7)--(10,10);
\draw[dashed](77/17,0)--(77/17,10);
\draw[dashed](0,0)--(10,10);
\draw[dashed](0.7,0.7)--(7.7,0.7);
\draw[dashed](0.7,0.7)--(0.7,7.7);
\draw[dashed](0.7,7.7)--(7.7,7.7);
\draw[dashed](7.7,0.7)--(7.7,7.7);
\foreach \x in {0}{\draw(\x,0)--(\x,0.1)node[below,outer sep=2pt,font=\tiny]at(\x,0){$0$};}
\foreach \x in {0.7}{\draw(\x,-0.1)--(\x,0.1)node[below,outer sep=2pt,font=\tiny]at(\x,0){$b$};}
\foreach \x in {77/17}{\draw(\x,-0.1)--(\x,0.1)node[below,outer sep=2pt,font=\tiny]at(\x,0){$\frac{b+1}{\beta}$};}
\foreach \x in {7.7}{\draw(\x,-0.1)--(\x,0.1)node[below,outer sep=2pt,font=\tiny]at(\x,0){$b+1$};}
\foreach \x in {10}{\draw(\x,0)--(\x,0.1)node[below,outer sep=2pt,font=\tiny]at(\x,0){$\frac{1}{\beta-1}$};}
\foreach \y in {0.7}{\draw(-0.1,\y)--(0.1,\y)node[left,outer sep=2pt,font=\tiny]at(0,\y){$b$};}
\foreach \y in {7.7}{\draw(-0.1,\y)--(0.1,\y)node[left,outer sep=2pt,font=\tiny]at(0,\y){$b+1$};}
\foreach \y in {10}{\draw(0,\y)--(0.1,\y)node[left,outer sep=2pt,font=\tiny]at(0,\y){$\frac{1}{\beta-1}$};}
\end{tikzpicture}
\caption{The graph of $T$.}
\end{figure}
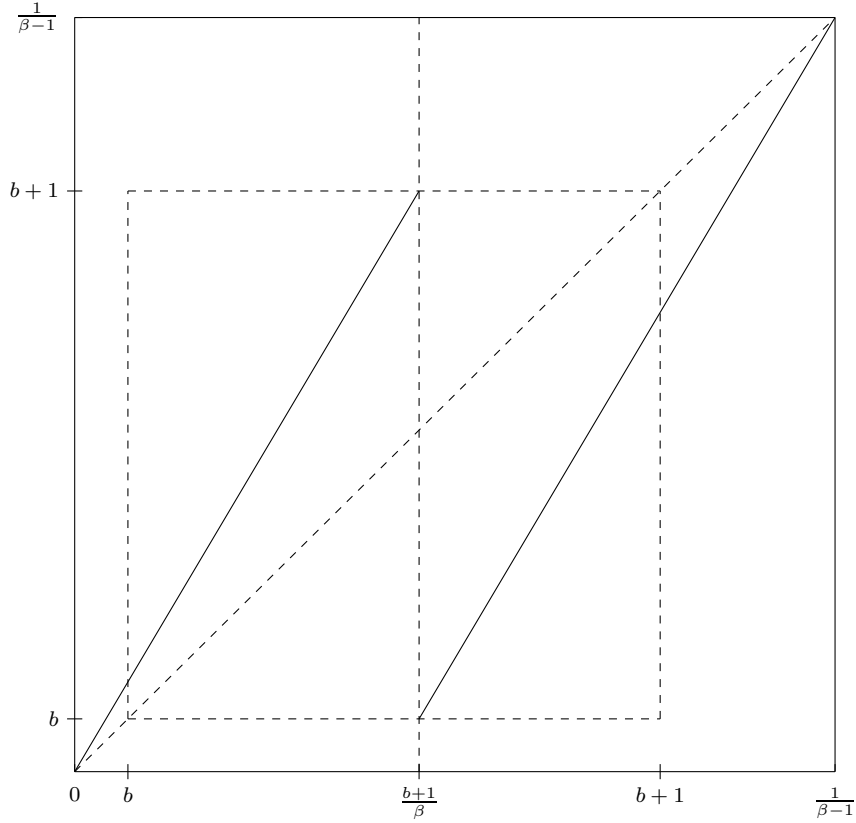
Noting that $T_0(\frac{b+1}{\beta})=b+1$ and $T_1(\frac{b+1}{\beta})=b$, by Section 3 in \cite{K90}, there exists a $T$-invariant ergodic measure $\mu\ll\cL$ (Lebesgue measure) on $ I_\beta$ such that for $\cL$-a.e. $x\in I_\beta$,
\begin{eqnarray}\label{density1}
\frac{d\mu}{d\cL}(x)=\sum_{n=0}^\infty\frac{\mathbbm{1}_{[0,T^n(b+1)]}(x)}{\beta^n}-\sum_{n=0}^\infty\frac{\mathbbm{1}_{[0,T^n(b)]}(x)}{\beta^n}
\end{eqnarray}
and $\nu:=\frac{1}{\mu(I_\beta)}\cdot\mu$ is a $T$-invariant ergodic probability measure on $ I_\beta$.
\begin{itemize}
\item[(1)] For $1\le n\le m-1$, prove $T^n(b)=\beta^nb<\frac{b+1}{\beta}\le\beta^nb+\beta^n-\beta^{n-1}-\cdots-\beta-1=T^n(b+1)$. Note that $\beta^m=\beta^{m-1}+\cdots+\beta+1=\frac{\beta^m-1}{\beta-1}$.
    \begin{itemize}
    \item[\textcircled{\footnotesize{$1$}}] By $b<\frac{1}{\beta-1}-1=\frac{1}{\beta^m-1}\le\frac{1}{\beta^{n+1}-1}$, we get $\beta^nb<\frac{b+1}{\beta}$.
    \item[\textcircled{\footnotesize{$2$}}] By $\frac{1}{\beta}+\cdots+\frac{1}{\beta^{n+1}}\le\frac{1}{\beta}+\cdots+\frac{1}{\beta^m}=1$, we get $\beta^n+\cdots+\beta+1\le\beta^{n+1}$ and then $\beta^n+\cdots+\beta+1+b\le\beta^{n+1}+\beta^{n+1}b$ which implies $\frac{b+1}{\beta}\le\beta^nb+\beta^n-\beta^{n-1}-\cdots-\beta-1$.
    \end{itemize}
\item[(2)] For $n\ge m$, prove $T^n(b)=T^n(b+1)$.
    \newline It suffices to prove $T^m(b)=T^m(b+1)$. In fact, this follows from (1) and $\beta^mb=\beta^mb+\beta^m-\beta^{m-1}-\cdots-\beta-1$.
\end{itemize}
Combining (\ref{density1}) and (2), we know that for $\cL$-a.e. $x\in I_\beta$,
\begin{eqnarray}\label{density2}
\frac{d\mu}{d\cL}(x)=\sum_{n=0}^{m-1}\frac{\mathbbm{1}_{[0,T^n(b+1)]}(x)-\mathbbm{1}_{[0,T^n(b)]}(x)}{\beta^n}.
\end{eqnarray}
Thus
\begin{eqnarray*}
\mu[0,\frac{b+1}{\beta})&=&\int_0^\frac{b+1}{\beta}\frac{d\mu}{d\cL}(x)dx\\
&=&\sum_{n=0}^{m-1}\frac{\min\{T^n(b+1),\frac{b+1}{\beta}\}-\min\{T^n(b),\frac{b+1}{\beta}\}}{\beta^n}\\
&\xlongequal[]{\text{by (1)}}&\sum_{n=0}^{m-1}\frac{\frac{b+1}{\beta}-\beta^nb}{\beta^n}\\
&=&1-(m-1)b
\end{eqnarray*}
where the last equality follows from $\frac{1}{\beta}+\cdots+\frac{1}{\beta^m}=1$. By
\begin{eqnarray*}
\mu(I_\beta)&=&\int_0^{\frac{1}{\beta-1}}\frac{d\mu}{d\cL}(x)dx\\
&=&\sum_{n=0}^{m-1}\frac{T^n(b+1)-T^n(b)}{\beta^n}\\
&\xlongequal[]{\text{by (1)}}&1+\sum_{n=1}^{m-1}\frac{\beta^n-\beta^{n-1}-\cdots-\beta-1}{\beta^n}\\
&=&1+\sum_{n=1}^{m-1}(1-\frac{1}{\beta}-\cdots-\frac{1}{\beta^n})\\
&=&m-\frac{m-1}{\beta}-\frac{m-2}{\beta^2}-\cdots-\frac{1}{\beta^{m-1}},
\end{eqnarray*}
we get $$\frac{1}{\beta}\cdot\mu(I_\beta)=\frac{m}{\beta}-\frac{m-1}{\beta^2}-\frac{m-2}{\beta^3}-\cdots-\frac{1}{\beta^m}.$$ It follows from the subtraction of the above two equalities that $\mu(I_\beta)=\frac{m\beta+\beta-2m}{\beta-1}$. Therefore $\nu=\frac{\beta-1}{m\beta+\beta-2m}\cdot\mu$ and
$$\nu[0,\frac{b+1}{\beta})=\frac{(\beta-1)(1-(m-1)b)}{m\beta+\beta-2m}=f(b)=p.$$
Since $T: I_\beta\to I_\beta$ is ergodic with respect to $\nu$, it follows from Birkhoff's Ergodic Theorem that for $\nu$-a.e. $x\in I_\beta$ we have
$$\lim_{n\to\infty}\frac{1}{n}\sum_{k=0}^{n-1}\mathbbm{1}_{[0,\frac{b+1}{\beta})}T^k(x)=\int_0^\frac{1}{\beta-1}\mathbbm{1}_{[0,\frac{b+1}{\beta})}d\nu=\nu[0,\frac{b+1}{\beta})=p,$$
which implies that for $\nu$-a.e. $x\in[b,b+1]$,
$$\lim_{n\to\infty}\frac{1}{n}\sum_{k=0}^{n-1}\mathbbm{1}_{[0,\frac{b+1}{\beta})}T^k(x)=p.$$
By (\ref{density2}) and (1), we know that for $\cL$-a.e. $x\in[b,b+1]$, $\frac{d\mu}{d\cL}(x)\ge1$. This implies $\cL\ll\mu(\sim\nu)$ on $[b,b+1]$, and then for $\cL$-a.e. $x\in[b,b+1]$, we have
$$\lim_{n\to\infty}\frac{1}{n}\sum_{k=0}^{n-1}\mathbbm{1}_{[0,\frac{b+1}{\beta})}T^k(x)=p.$$
For every $x\in I_\beta$, define a sequence $(\epsilon_i(x))_{i\ge1}\in\{0,1\}^\N$ by
$$\epsilon_i(x):=\left\{\begin{array}{ll}
0 & \text{if } T^{i-1}x\in[0,\frac{b+1}{\beta})\\
1 & \text{if } T^{i-1}x\in[\frac{b+1}{\beta},\frac{1}{\beta-1}]
\end{array}\right.\quad\text{for all }i\ge1.$$
Then by
$$\mathbbm{1}_{[0,\frac{b+1}{\beta})}(T^kx)=1\Leftrightarrow T^kx\in[0,\frac{b+1}{\beta})\Leftrightarrow\epsilon_{k+1}(x)=0,$$
we know that for $\cL$-a.e. $x\in[b,b+1]$,
\begin{eqnarray}\label{frequency}
\lim_{n\to\infty}\frac{\sharp\{1\le i\le n:\epsilon_i(x)=0\}}{n}=p,\quad\text{i.e.,}\quad\text{Freq}_0(\epsilon_i(x))=p.
\end{eqnarray}
By the same way as in the proof of Lemma \ref{lemma}, we know that for every $x\in I_\beta$, the $(\epsilon_i(x))_{i\ge1}$ defined above is a $\beta$-expansion of $x$, and Lebesgue almost every $x\in I_\beta$
has a $\beta$-expansion with frequency of zeros equal to $p$. Then we finish the proof by applying Theorem \ref{main1}.
\end{proof}

\section{further questions}

First we wonder whether Theorem \ref{main1} can be generalized.

\begin{question}\label{Q1}
Let $\beta\in(1,+\infty)\setminus\N$ and $\overline{p},\underline{p}\in[0,1]^{\lceil\beta\rceil}$. Is it true that Lebesgue almost every $x\in I_\beta$ has a $\beta$-expansion of frequency $(\overline{p},\underline{p})$ if and only if Lebesgue almost every $x\in I_\beta$ has a continuum of $\beta$-expansions of frequency $(\overline{p},\underline{p})$?
\end{question}

If a positive answer is given to this question, by Theorem \ref{main1} and \ref{main2}, there is also a positive answer to the following question.

\begin{question}\label{Q2}
Let $\beta\in(2,+\infty)\setminus\N$. Is it true that Lebesgue almost every $x\in I_\beta$ has a continuum of balanced $\beta$-expansions?
\end{question}

Even if a negative answer is given to Question \ref{Q1}, there may be a positive answer to Question \ref{Q2}. An intuitive reason is that, when $\beta>2$, we have $\sharp\cA_\beta\ge3$ and balanced $\beta$-expansions are much more flexible than simply normal $\beta$-expansions.

The last question we want to ask is on the variability of the frequency related to Theorem \ref{main3}. Let $\beta>1$. If there exists $c=c(\beta)>0$ such that for any $p_0,p_1,\cdots,p_{\lceil\beta\rceil-1}\in[\frac{1}{\lceil\beta\rceil}-c,\frac{1}{\lceil\beta\rceil}+c]$ with $p_0+p_1+\cdots+p_{\lceil\beta\rceil-1}=1$, every $x\in I_\beta^o$ has a $\beta$-expansion $(\epsilon_i)_{i\ge1}$ with
$$\text{Freq}_0(\epsilon_i)=p_0,\text{ Freq}_1(\epsilon_i)=p_1,\cdots,\text{ Freq}_{\lceil\beta\rceil-1}(\epsilon_i)=p_{\lceil\beta\rceil-1},$$
we say that $\beta$ is a \textit{variational frequency} base. Similarly, if there exists $c=c(\beta)>0$ such that for any $p_0,p_1,\cdots,p_{\lceil\beta\rceil-1}\in[\frac{1}{\lceil\beta\rceil}-c,\frac{1}{\lceil\beta\rceil}+c]$ with $p_0+p_1+\cdots+p_{\lceil\beta\rceil-1}=1$, Lebesgue almost every $x\in I_\beta$ has a $\beta$-expansion $(\epsilon_i)_{i\ge1}$ with
$$\text{Freq}_0(\epsilon_i)=p_0,\text{ Freq}_1(\epsilon_i)=p_1,\cdots,\text{ Freq}_{\lceil\beta\rceil-1}(\epsilon_i)=p_{\lceil\beta\rceil-1},$$
we say that $\beta$ is an \textit{almost variational frequency} base.

Obviously, all variational frequency bases are almost variational frequency bases. Baker's results (see the statements between Theorem \ref{main2} and Theorem \ref{main3}) say that all numbers in $(1,\frac{1+\sqrt{5}}{2})$ are variational frequency bases and all numbers in $[\frac{1+\sqrt{5}}{2},2)$ are not variational frequency bases. Fortunately, Theorem \ref{main3} says that pseudo-golden ratios (which are all in $[\frac{1+\sqrt{5}}{2},2)$) are almost variational frequency bases. We wonder whether all numbers in $[\frac{1+\sqrt{5}}{2},2)$ are almost variational frequency bases.

For all integers $\beta>1$, we know that Lebesgue almost every $x\in[0,1]$ has a unique $\beta$-expansion $(\epsilon_i)_{i\ge1}$, and this expansion satisfies
$$\text{Freq}_0(\epsilon_i)=\text{Freq}_1(\epsilon_i)=\cdots=\text{Freq}_{\beta-1}(\epsilon_i)=\frac{1}{\beta}$$
by Borel's normal number theorem. Therefore all integers are not almost variational frequency bases. It is natural to ask the following question.

\begin{question}\label{Q3}
Is it true that all non-integers greater than $1$ are almost variational frequency bases?
\end{question}

A positive answer is expected.

\begin{ack}
The author is grateful to Professor Jean-Paul Allouche for his advices on a former version of this paper, and also grateful to the Oversea Study Program of Guangzhou Elite Project.
\end{ack}


\begin{thebibliography}{00}

\bibitem[1]{ACS09} \textsc{J.-P. Allouche, M. Clarke and N. Sidorov}, \textit{Periodic unique beta-expansions: the Sharkovski$\breve{i}$ ordering}, Ergodic Theory Dynam. Systems 29 (2009), no. 4, 1055--1074.

\bibitem[2]{AC01} \textsc{J.-P. Allouche and M. Cosnard}, \textit{Non-integer bases, iteration of continuous real maps, and an arithmetic self-similar set}, Acta Math. Hungar. 91 (2001), no. 4, 325--332.

\bibitem[3]{B14} \textsc{S. Baker}, \textit{Generalised golden ratios over integer alphabets}, Integers 14 (2014), Paper No. A15, 28 pp.

\bibitem[4]{B18} \textsc{S. Baker}, \textit{Digit frequencies and self-affine sets with non-empty interior}, Ergodic Theory Dynam. Systems (First published online 2018), 1--33.

\bibitem[5]{BK18} \textsc{S. Baker, D.Kong}, \textit{Numbers with simply normal $\beta$-expansions}, Math. Proc. Cambridge Philos. Soc. 167 (2019), no. 1, 171--192.

\bibitem[6]{BS14} \textsc{S. Baker, N. Sidorov}, \textit{Expansions in non-integer bases: lower order revisited}, Integers 14 (2014), Paper No. A57, 15 pp.

\bibitem[7]{B89} \textsc{F. Blanchard}, \textit{$\beta$-expansions and symbolic dynamics}, Theoret. Comput. Sci. 65 (1989), no. 2, 131--141.

\bibitem[8]{B09} \textsc{E. Borel}, \textit{Les probabilit\'es d\'enombrables et leurs applications arithm\'etiques}, Rend. Circ. Mat. Palermo (2) 27 (1909) 247--271.

\bibitem[9]{BL16} \textsc{Y. Bugeaud and L. Liao}, \textit{Uniform Diophantine approximation related to $b$-ary and ${\it\beta}$-expansions}, Ergodic Theory Dynam. Systems 36 (2016), no. 1, 1--22.

\bibitem[10]{DK09} \textsc{M. de Vries, V. Komornik}, \textit{Unique expansions of real numbers}, Adv. Math. 221 (2009), no. 2, 390--427.

\bibitem[11]{E49} \textsc{H. Eggleston}, \textit{The fractional dimension of a set defined by decimal properties}, Quart. J. Math., Oxford Ser. 20, (1949). 31--36.

\bibitem[12]{EJK90} \textsc{P. Erd\"os, I. Jo\'o and V. Komornik}, \textit{Characterization of the unique expansions $1=\sum^{\infty} _ {i= 1} q^{-n_ i} $ and related problems}, Bull. Soc. Math. France 118 (1990), no. 3, 377--390.

\bibitem[13]{EJK94} \textsc{P. Erd\"os, I. Jo\'o and V. Komornik}, \textit{On the number of $q$-expansions}, Ann. Univ. Sci. Budapest. E\"otv\"os Sect. Math. 37 (1994), 109--118.

\bibitem[14]{F14} \textsc{K. J. Falconer}, \textit{Fractal geometry}, Mathematical foundations and applications. Third edition. John Wiley $\&$ Sons, Ltd., Chichester, 2014. xxx+368 pp.

\bibitem[15]{FW12} \textsc{A.-H. Fan and B.-W. Wang}, \textit{On the lengths of basic intervals in beta expansions}, Nonlinearity 25 (2012), no. 5, 1329--1343.

\bibitem[16]{FS92} \textsc{C. Frougny, B. Solomyak}, \textit{Finite beta-expansions}, Ergodic Theory Dynam. Systems 12 (1992), no. 4, 713--723.

\bibitem[17]{GS01} \textsc{P. Glendinning, N. Sidorov}, \textit{Unique representations of real numbers in non-integer bases}, Math. Res. Lett. 8 (2001), no. 4, 535--543.

\bibitem[18]{JSS11} \textsc{T. Jordan, P. Shmerkin, B. Solomyak}, \textit{Multifractal structure of Bernoulli convolutions}, Math. Proc. Cambridge Philos. Soc. 151 (2011), no. 3, 521--539.

\bibitem[19]{KL14} \textsc{D. Kong, W. Li}, \textit{Hausdorff dimension of unique beta expansions}, Nonlinearity 28 (2015), no. 1, 187--209.

\bibitem[20]{K90} \textsc{C. Kopf}, \textit{Invariant measures for piecewise linear transformations of the interval}, Appl. Math. Comput. 39 (1990), no. 2, part II, 123--144.

\bibitem[21]{LW08} \textsc{B. Li and J. Wu}, \textit{Beta-expansion and continued fraction expansion}, J. Math. Anal. Appl. 339 (2008), no. 2, 1322--1331.

\bibitem[22]{LL18} \textsc{Y.-Q. Li and B. Li}, \textit{Distributions of full and non-full words in beta-expansions}, J. Number Theory 190 (2018), 311--332.

\bibitem[23]{P60} \textsc{W. Parry}, \textit{On the $\beta$-expansions of real numbers}, Acta Math. Acad. Sci. Hungar. 11 (1960), 401--416.

\bibitem[24]{R57} \textsc{A. R\'enyi}, \textit{Representations for real numbers and their ergodic properties}, Acta Math. Acad. Sci. Hungar. 8 (1957), 477--493.

\bibitem[25]{S97} \textsc{J. Schmeling}, \textit{Symbolic dynamics for $\beta$-shifts and self-normal numbers}, Ergodic Theory Dynam. Systems 17 (1997), no. 3, 675--694.

\bibitem[26]{S80} \textsc{K. Schmidt}, \textit{On periodic expansions of Pisot numbers and Salem numbers}, Bull. London Math. Soc. 12 (1980), no. 4, 269--278.

\bibitem[27]{S03} \textsc{N. Sidorov}, \textit{Almost every number has a continuum of $\beta$-expansions}, Amer. Math. Monthly 110 (2003), no. 9, 838--842.

\bibitem[28]{W82} \textsc{P. Walters}, \textit{An Introduction to Ergodic Theory}, Graduate Texts in Mathematics, 79. Springer-Verlag, New York-Berlin, 1982. ix+250 pp.

\bibitem[29]{W75} \textsc{K. M. Wilkinson}, \textit{Ergodic properties of a class of piecewise linear transformations}, Z. Wahrscheinlichkeits\-theorie und Verw. Gebiete 31 (1974/75), 303--328.

\end{thebibliography}
\end{document}